\documentclass[12pt,a4paper]{article}%
\usepackage[utf8]{inputenc}
\usepackage{cite}
\usepackage{hyperref}
\usepackage{amsmath}
\usepackage{amsfonts}
\usepackage{amssymb}
\usepackage{xcolor}
\usepackage{graphicx}
\usepackage{grffile}%
\providecommand{\U}[1]{\protect\rule{.1in}{.1in}}
\newtheorem{theorem}{Theorem}

\newtheorem{conjecture}[theorem]{Conjecture}
\newtheorem{corollary}[theorem]{Corollary}

\newtheorem{proposition}[theorem]{Proposition}

\newenvironment{proof}[1][Proof]{\noindent\textbf{#1.} }{\ \hfill \rule{0.5em}{0.5em}\bigskip}
\graphicspath{{Slike/}{d:/Dropbox/Riste-Sedlar/2026.07.08 Queen conjecture/Slike/}}

\begin{document}

\title{Odd independence number of the Queen graph}
\author{Martin Knor$^{1}$, Jelena Sedlar$^{2,4}$, Riste \v{S}krekovski$^{3,4,5}$\\{\small $^{1}$ \textit{Slovak University of Technology in Bratislava,
Slovakia}}\\[0.1cm] {\small $^{2}$ \textit{University of Split, Faculty of civil
engineering, architecture and geodesy, Croatia}}\\[0.1cm] {\small $^{3}$ \textit{University of Ljubljana, Faculty of Mathematics
and Physics, Slovenia}}\\[0.1cm] {\small $^{4}$ \textit{University of Novo mesto, Faculty of
Information Studies, Slovenia}}\\[0.1cm] {\small $^{5}$ \textit{Rudolfovo -- Science and Technology Centre Novo
mesto, Slovenia}}}
\date{}
\maketitle

\begin{abstract}
A set $S$ of vertices of a graph is \emph{odd independent} if it is
independent and every vertex outside $S$ has either zero or an odd number of
neighbors in $S$. The largest size of such a set is the odd independence
number $\alpha_{\mathrm{od}}$. Caro, Petru\v{s}evski, \v{S}krekovski and Tuza
conjectured that $\alpha_{\mathrm{od}}=1$ for every finite Queen graph. They
also asked whether the infinite Queen graph has $\alpha_{\mathrm{od}}=1$ or
$\alpha_{\mathrm{od}}=\infty$. We prove that $\alpha_{\mathrm{od}}=1$ in both
cases. In particular, in the case of an infinite board we prove that
$\alpha_{\mathrm{od}}=1$ holds on the quarter plane as well as on the whole plane.

\end{abstract}

\textit{Keywords:} odd independence number; strong odd chromatic number; Queen
graph; chessboard graphs.

\textit{AMS Subject Classification numbers:} 05C69, 05C15

\section{Introduction}

\label{sec:uvod}

Let $G=(V,E)$ be a graph, finite or infinite, and let $N(v)$ denote the open
neighborhood of a vertex $v$. A set $S\subseteq V$ is an \emph{odd independent
set} if it is independent and, for every vertex $v\not \in S$, the set
$N(v)\cap S$ is empty or of an odd size. The largest size of such a set is the
\emph{odd independence number} $\alpha_{\mathrm{od}}(G)$, introduced
in~\cite{CPST1}.

A proper coloring of $G$ is a \emph{strong odd coloring} if, for every vertex
$v$, each color occurring in $N(v)$ occurs there an odd number of times. The
least number of colors in such a coloring is the \emph{strong odd chromatic
number} $\chi_{\mathrm{so}}(G)$~\cite{CPSTso}. The odd independence number is
its counterpart on the independence side. Strong odd colorings have received
considerable attention recently~\cite{GKKPPU,KP,PMF}. Every color class of a
strong odd coloring is an odd independent set, so the two parameters satisfy
\[
\alpha_{\mathrm{od}}(G)\cdot\chi_{\mathrm{so}}(G)\geq\left\vert
V(G)\right\vert .
\]
Therefore a graph with $\alpha_{\mathrm{od}}(G)=1$ has the largest possible
value of $\chi_{\mathrm{so}}(G)$, namely $\left\vert V(G)\right\vert $.

The chessboard graphs are defined on the vertex set $\{1,\ldots,n\}\times
\{1,\ldots,n\}$, where the adjacencies correspond to the moving rules of
chess. We call $n$ the \emph{size} of the graph, after the side length of the
board on which it is taken. We use the cellular representation which
identifies the $n^{2}$ vertices with the $n^{2}$ cells by their coordinate
pairs. A vertex $(a,b)$ is the cell in the intersection of column $a$ and row
$b$. In the \emph{Queen graph}, two cells are adjacent whenever a queen moves
from one to the other, that is, $(a,b)\sim(a^{\prime},b^{\prime})$ if and only
if $a=a^{\prime}$, or $b=b^{\prime}$, or $a-a^{\prime}=b-b^{\prime}$, or
$a-a^{\prime}=b^{\prime}-b$. Keeping only the first two adjacencies gives the
\emph{Rook graph} $K_{n}\square K_{n}$, and keeping only the last two gives
the \emph{Bishop graph}. The Queen graph on a $p\times q$ board has vertex set
$\{1,\ldots,p\}\times\{1,\ldots,q\}$ and the same adjacency rule. The
\emph{infinite} Queen graph is given by that rule on the quarter plane
$\mathbb{N}\times\mathbb{N}$ or on the whole plane $\mathbb{Z}\times
\mathbb{Z}$.

A set of vertices is independent exactly when it is a set of pairwise
non-attacking queens. Every vertex of the Queen graph lies on exactly four
lines, its row, its column, its diagonal and its
antidiagonal. Each of them contains at
most one element of an independent set $S$.
A line is \emph{occupied} if it contains an element
of $S$, so a vertex $v\not \in S$ is attacked by precisely one element of $S$
along every occupied line through $v$. Hence $\left\vert N(v)\cap
S\right\vert \leq4$ for every vertex $v$. Therefore a set $S$ of
pairwise non-attacking queens is an odd independent set if and only if every
vertex $v\not \in S$ lies on $0$, $1$ or $3$
occupied lines.

Throughout the rest of the paper, elements of a finite odd independent set $S$
of cardinality $k$ are denoted by $(a_{i},b_{i})$ for $1\leq i\leq k$. The
independence of $S$ means that $a_{i}\not =a_{j}$, $b_{i}\not =b_{j}$ and
$\left\vert a_{i}-a_{j}\right\vert \not =\left\vert b_{i}-b_{j}\right\vert $
hold for all $i\not =j$. We call $w=\max_{i}a_{i}-\min_{i}a_{i}$ the
\emph{width} and $h=\max_{i}b_{i}-\min_{i}b_{i}$ the \emph{height} of $S$.

For the classical domination and independence problems on the chessboard we
refer the reader to~\cite{Harris}. The odd independence for chessboard graphs
was investigated systematically by Caro, Petru\v{s}evski, \v{S}krekovski and
Tuza~\cite{CPST2}. They determined $\alpha_{\mathrm{od}}$ and $\chi
_{\mathrm{so}}$ for the Rook, the Bishop, and the $r$-King graph on a board of
any size, and also on several families of grids. A value as small as $1$ is
usually not expected for a chessboard graph. The king moves are exactly the
queen moves of one step, yet the King graph of size $n$ has odd independence
number $\left\lceil n/3\right\rceil ^{2}$, see~\cite{CPST2}.

Their investigation left the Queen graph undecided, since the three tools used
in~\cite{CPST2} do not apply to it: although the graph is $K_{1,5}$-free, it
is far from regular; although its diameter is $2$, it is not claw-free; and
the infinite Queen graph has no forbidden pair and no forcing
pair~\cite{CPST1}. Only the boards with $n\leq10$ were settled by computer
search, and all of them have $\alpha_{\mathrm{od}}=1$. For the infinite board
the following is known, see~\cite{CPST2}.

\begin{proposition}
\label{Prop_dihotomija}
The infinite Queen graph has either $\alpha_{\mathrm{od}}=\infty$ or $\alpha_{\mathrm{od}}=1$, but no other finite value is possible.
\end{proposition}

\begin{proof}
[Sketch of the argument in~\cite{CPST2}]
Assume to the contrary that
$1<k<\infty$, and let $S$ be an odd independent set of size $k$. Let
$(a_{1},b_{1})\in S$ be a vertex whose first coordinate is largest, and let
$(a_{2},b_{2})\in S$ be one for which $\left\vert b_{1}-b_{2}\right\vert $ is
largest. One then considers the vertex $v=(a_{0},b_{2})$, where $a_{0}%
=a_{1}-b_{1}+b_{2}$ if $b_{2}>b_{1}$ and $a_{0}=a_{1}+b_{1}-b_{2}$ if
$b_{2}<b_{1}$. Thus $v$ lies on the row of $(a_{2},b_{2})$ and on a diagonal
of $(a_{1},b_{1})$. This $v$ has \emph{exactly two} neighbors in $S$, namely
$(a_{2},b_{2})$ horizontally and $(a_{1},b_{1})$ diagonally, because the
column and the other diagonal through $v$ are
not occupied. Since $2$ is
even, this contradicts the odd independence of $S$.
\end{proof}

This argument uses the \emph{infinite} board only to guarantee that the
constructed vertex $v$ lies on it. On a finite $n\times n$ board $v$ could
fall outside $\{1,\ldots,n\}^{2}$, and closing this gap is what the following
conjecture asks for, see~\cite{CPST2}.

\begin{conjecture}
\label{Con_kraljica}
The Queen graph of any finite size has $\alpha_{\mathrm{od}}=1$.
\end{conjecture}

For the infinite board, it was asked in~\cite{CPST2} which of the two
alternatives of Proposition \ref{Prop_dihotomija} occurs.

In this paper we prove Conjecture \ref{Con_kraljica} and we answer the
question for the infinite board. In Section
\ref{sec:konacni} the set $S$ is finite, and we count the diagonal attacks
along the four sides of its bounding box. The count forces the box to be a
square with two elements of $S$ in its opposite corners, which is impossible.
This proves Conjecture \ref{Con_kraljica}. In Section
\ref{sec:beskonacni} the set $S$ is infinite, and we count the lines
containing elements of $S$ through the vertices of closed zigzag paths between
the rows of three elements of $S$. This time the count forces infinitely many
elements of $S$ to share a column, a diagonal or an antidiagonal, which is
impossible as well. This answers the question for the infinite board, and the
argument settles the quarter plane and the whole plane alike.

Authors of \cite{CPST2} independently sketched a proof of Theorem
\ref{Tm_nemaBeskonacnog} based on an exhaustive case analysis, but they did
not elaborate it in detail as doing so would significantly extend the paper.

\section{Finite odd independent sets}

\label{sec:konacni}

Throughout this section $S$ is finite.
We first prove the first of the two main results of this paper.
The proof counts the diagonal attacks along the four
sides of the bounding box of $S$, see Figure \ref{Fig02}. Every rectangular
board carrying $S$ contains this box, so the argument does not depend on the board.

\begin{figure}[tbh]
\begin{center}
\includegraphics[scale=1]{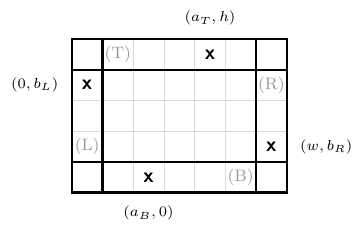}
\end{center}
\caption{The bounding box $[0,w]\times\lbrack0,h]$ of $S$ and its four sides
$(\mathrm{B})$, $(\mathrm{T})$, $(\mathrm{L})$, $(\mathrm{R})$, along which
the diagonal attacks are counted, with \textsf{\bfseries x} marking an element
of $S$.}%
\label{Fig02}%
\end{figure}

\begin{theorem}
\label{Tm_nemaKonacnog}
The Queen graph of any size, finite or infinite, has
no odd independent set $S$ with $2\leq\left\vert S\right\vert <\infty$.
\end{theorem}

\begin{proof}
Assume to the contrary that $S$ is an odd independent set with $k=\left\vert
S\right\vert $ and $2\leq k<\infty$. Due to the symmetry of the board we may
assume $w\geq h.$ Also, we may assume that $\min_{i}a_{i}=\min_{i}b_{i}=0$, so
that $w=\max_{i}a_{i}$ and $h=\max_{i}b_{i}$.
Since the $k\geq2$ queens of $S$ lie in $k$ distinct rows, we have $h\geq1$, hence $w\geq1$. Let the finite
\emph{box} $\mathcal{B}$ be the set of all cells $(x,y)$ of the board with
$0\leq x\leq w$ and $0\leq y\leq h.$ Its \emph{left}, \emph{right},
\emph{bottom} and \emph{top} \emph{sides} are the sets $L,$ $R,$ $B,$ $T$ of
all $(x,y)\in\mathcal{B}$ with $x=0,$ $x=w,$ $y=0$ and $y=h$ respectively, and
their union is the \emph{border} of $\mathcal{B}$. Since
$\mathcal{B}$ is the smallest finite box which contains $S,$ each side of
$\mathcal{B}$ must contain a queen. Let $(a_{B},0)$, $(a_{T},h)$, $(0,b_{L})$
and $(w,b_{R})$ be the queens of $S$ lying in the bottom, top, left and the
right side of $\mathcal{B}$, as in Figure \ref{Fig02}. These are also called
the \emph{border queens}.

\medskip

\noindent\textbf{Claim A.}
\emph{No cell from the border of $\mathcal{B}$ is attacked by queens of $S$ along both diagonals.}

\noindent Let $H$ (resp. $V$) be the set of all cells $(x,y)\not \in S$ of the
board attacked by a queen horizontally (resp. vertically), i.e. the row $y$
(resp. the column $x$) is occupied. Consider a side $X\in\{L,R,B,T\}$
of the box. Notice that $\left\vert X\cap(H\cap V)\right\vert =k-1,$ since
each cell of the side $X$ distinct from the border queen is attacked by the
border queen, and there are $k-1$ remaining queens in $S$ each of which
attacks one cell of the side. Since $S$ is an odd independent set, and each
cell of $X\cap(H\cap V)$ is attacked already by two queens, for each of the
$k-1$ cells in $X\cap(H\cap V)$ there must exist precisely one queen of $S$
attacking it diagonally. The cells of $X\backslash(H\cap V)$ are attacked
either by no queen diagonally or by two.

An instance of a diagonal attack is a cell paired with one queen of $S$
attacking it along a diagonal, so a cell attacked along both diagonals carries
two instances.
For a side $X\in\{L,R,B,T\},$ let $d_{X}$ denote the number of instances of a diagonal attack on $X$.
We count the number of attacks in the border.
Each queen of $S$ which is not on the border of $\mathcal{B}$ contributes to this count by $4$ instances.
A queen on the border of $\mathcal{B}$ which is not in a corner contributes by $2$ instances, and a queen in a corner contributes by $1$ instance.
So denote by $r$ the number of queens on the border of $\mathcal{B}$ which are not in a corner and denote by $s$ the number of queens in the corners.
They together produce $4(k-r-s)+2r+s$ attacks.
But when counting $d_L+d_R+d_B+d_T$, the attacks in corners are counted twice.
So if a corner is not occupied by a queen, we have to add an extra attack to the sum, since such a corner is always in $X\cap(H\cap V)$ due to the border queens, hence it carries exactly one instance.
Hence, in $d_L+d_R+d_B+d_T$ we have
$$
4(k-r-s)+2r+s+(4-s)=4k+4-4s-2r
$$
attacks.

Now observe that $s\in\{0,1,2\}$, since the queens cannot attack each other.
Moreover $s$ determines $r$, and the pairs $(s,r)$ can be only $(0,4)$, $(1,2)$ and $(2,0)$.
And in each of them
$$
4k+4-4s-2r=4k-4,
$$
which equals the sum of $\left\vert X\cap(H\cap V)\right\vert$ over all four sides $X$.
Consequently, no cell of the border is attacked diagonally by $2$ queens.

\medskip

\noindent\textbf{Claim B.}
\emph{The box $\mathcal{B}$ must be a square, i.e. $w=h$.}

\noindent Assume to the contrary that $w>h$.
Then no queen of $S$ attacks a cell on both the left and the right side
along its diagonal, and the same holds along its antidiagonal.
Moreover, the queens at $(0,b_L)$ and $(w,b_R)$ contribute by no attack to $L$ and $R$.
Consequently, in $L$ and $R$ we have at most $2(k-2)=2k-4$ attacks while
$|(L\cup R)\cap(H\cap V)|=2k-2$.
So $S$ cannot be an odd independent set.

\medskip

\noindent\textbf{Claim C.}
\emph{The cells $(0,0)$ and $(w,w)$ both contain a queen of $S$.}

\noindent By Claim B we have $w=h$. Let $U_{1}$ (resp. $U_{2}$) be the set of
queens $(a,b)\in S$ with $a\geq b$ (resp. $a\leq b$), and let $U_{3}$ (resp.
$U_{4}$) be the set of queens $(a,b)\in S$ with $a+b\leq w$ (resp. $a+b\geq
w$). Thus $U_{1}\cup U_{2}=U_{3}\cup U_{4}=S$, and the queens in $U_{1}\cap
U_{2}$ are those with $a=b$. Notice that the instances of a diagonal attack on
the bottom side are given by the queens of $U_{1}$ along their diagonals and
by the queens of $U_{3}$ along their antidiagonals, except for the border queen
$(a_{B},0)\in U_{1}\cap U_{3}$, so $d_{B}=\left\vert
U_{1}\right\vert +\left\vert U_{3}\right\vert -2$. Likewise $d_{L}=\left\vert
U_{2}\right\vert +\left\vert U_{3}\right\vert -2$. By Claim A, every side $X$
has $d_{X}=\left\vert X\cap(H\cap V)\right\vert =k-1$. Since $d_{B}=d_{L}=k-1$, we
obtain $\left\vert U_{1}\right\vert =\left\vert U_{2}\right\vert $.

Now we use Claim A.
On the bottom side, the cells in $H\cap V$ are the cells $(a,0)$ with the column $a$ occupied and $a\not =a_{B}$, each carrying exactly one instance of a diagonal attack, and every other cell carries none.
So summing the first coordinates of the
cells over all instances of a diagonal attack on the bottom side gives
$\sum_{i=1}^{k}a_{i}-a_{B}$, while by the above description of the instances
the same sum equals $\sum_{(a,b)\in U_{1}\setminus\{(a_B,0)\}}(a-b)+\sum_{(a,b)\in U_{3}\setminus\{(a_B,0)\}}(a+b)$.
For easier manipulation with the equation we add to both sides twice $a_B$, which equals $a-b$ and also $a+b$ for $(a_B,0)$.
We proceed in the same way on the other three sides, summing the first
coordinates on the top side and the second coordinates on the left and the
right side.
This gives
\begin{align*}
(\mathrm{B})\quad\sum_{(a,b)\in U_{1}}(a-b)+\sum_{(a,b)\in U_{3}}(a+b)  &
=\sum_{i=1}^{k}a_{i}+a_{B},\\
(\mathrm{L})\quad\sum_{(a,b)\in U_{2}}(b-a)+\sum_{(a,b)\in U_{3}}(a+b)  &
=\sum_{i=1}^{k}b_{i}+b_{L},\\
(\mathrm{T})\quad\sum_{(a,b)\in U_{2}}(a-b+w)+\sum_{(a,b)\in U_{4}}(a+b-w)  &
=\sum_{i=1}^{k}a_{i}+a_{T},\\
(\mathrm{R})\quad\sum_{(a,b)\in U_{1}}(b-a+w)+\sum_{(a,b)\in U_{4}}(a+b-w)  &
=\sum_{i=1}^{k}b_{i}+b_{R}.
\end{align*}
Subtracting (L) from (B) we obtain
\[
\sum_{(a,b)\in U_{1}}(a-b)+\sum_{(a,b)\in U_{2}}(a-b)=\sum_{i=1}^{k}%
(a_{i}-b_{i})+a_{B}-b_{L}.
\]
A queen $(a,b)$ lying in both $U_{1}$ and $U_{2}$ satisfies $a-b=0$, so the
left-hand side equals $\sum_{i=1}^{k}(a_{i}-b_{i})$, and therefore
$a_{B}=b_{L}$. Subtracting (T) from (R) and using $\left\vert U_{1}\right\vert
=\left\vert U_{2}\right\vert $ we obtain in the same way
\[
\sum_{(a,b)\in U_{1}}(b-a)+\sum_{(a,b)\in U_{2}}(b-a)=\sum_{i=1}^{k}%
(b_{i}-a_{i})+b_{R}-a_{T},
\]
and therefore $b_{R}=a_{T}$. Now the border queens $(a_{B},0)$ and $(0,b_{L})$
lie on the same antidiagonal $x+y=a_{B}=b_{L}$. Since $S$ is independent, they
are the same queen, which gives $a_{B}=b_{L}=0$ and $(0,0)\in S$. Analogously,
the border queens $(a_{T},w)$ and $(w,b_{R})$ lie on the same antidiagonal
$x+y=a_{T}+w=b_{R}+w$, so they are the same queen, which gives $a_{T}=b_{R}=w$
and $(w,w)\in S$. This proves Claim C.

\medskip

By Claim C the cells $(0,0)$ and $(w,w)$ both contain a queen of $S$, and
since $w\geq1$ these two queens are distinct. But they lie on the same
diagonal $x-y=0$, so they attack each other, contradicting the independence of
$S$.
\end{proof}

Next corollary is a direct consequence of Theorem~{\ref{Tm_nemaKonacnog}}.

\begin{corollary}
\label{Cor_konjektura}
Every finite Queen graph has $\alpha_{\mathrm{od}}=1$.
Consequently the Queen graph of size $n$ has $\chi_{\mathrm{so}}=n^{2}$.
Moreover, every odd independent set of the infinite Queen graph is empty, a
singleton, or infinite.
\end{corollary}

So Conjecture \ref{Con_kraljica} is proved. Notice that the proof of Theorem
\ref{Tm_nemaKonacnog} uses only that the board is a rectangle. Hence the
theorem holds unchanged for the Queen graph on a $p\times q$ board, and so
does Corollary \ref{Cor_konjektura}, with $\alpha_{\mathrm{od}}=1$ and
$\chi_{\mathrm{so}}=pq$.

\section{Infinite odd independent sets}

\label{sec:beskonacni}

By Corollary \ref{Cor_konjektura}, the infinite Queen graph has either
$\alpha_{\mathrm{od}}=1$ or $\alpha_{\mathrm{od}}=\infty$.
In this section we prove that the former alternative holds on the quarter plane
$\mathbb{N}\times\mathbb{N}$, as well as on the whole plane $\mathbb{Z}%
\times\mathbb{Z}$. The following theorem is the second main result of this
paper.

\begin{theorem}
\label{Tm_nemaBeskonacnog}
The Queen graph on an infinite board, be it the quarter plane $\mathbb{N}\times\mathbb{N}$ or the whole plane $\mathbb{Z}\times\mathbb{Z}$, has no infinite odd independent set.
\end{theorem}

\begin{proof}
Assume to the contrary that $S$ is an infinite odd independent set. Since $S$
is infinite and every row and column contain at most one queen, on the quarter
plane $\mathbb{N}\times\mathbb{N}$ we may assume that $S$ has queens in
arbitrarily high rows. Due to symmetry of the board, we may assume the same on
$\mathbb{Z}\times\mathbb{Z}$.

For a cell $c$, let $\sigma(c)$ denote the number of occupied lines among the
column, the diagonal and the antidiagonal of $c$.
Clearly, if $c\in S$ then $\sigma(c)=3$. Also,
if $c\not \in S$ and $c$ lies in an occupied row,
then $\sigma(c)$ must be even.

Fix three queens of $S$ with rows $b_{1}<b_{2}<b_{3}$, and let $a^{\ast}$ be
the rightmost of their columns. The \emph{window} of a cell $c,$ illustrated
by Figure \ref{Fig03}, is the list of the six cells of the row of $c$ at the
horizontal distances
\[
0,\quad b_{2}-b_{1},\quad b_{3}-b_{2},\quad b_{2}+b_{3}-2b_{1},\quad
2b_{3}-b_{1}-b_{2},\quad2b_{3}-2b_{1}
\]
to the right of $c$, counted with multiplicity.
We emphasise that windows defined this way, i.e. by the three queens in rows $b_1$, $b_2$ and $b_3$, will be used for all other rows occupied by queens of $S$.

For a cell $c$, let
$N_{\mathrm{col}}(c)$ (resp. $N_{\mathrm{dg}}(c)$, $N_{\mathrm{ad}}(c)$)
denote the number of cells of the window of $c$ lying in occupied columns
(resp. on occupied diagonals, on occupied antidiagonals).
Moving $c$ along a column (resp. diagonal, antidiagonal) $L$ moves every cell of its window along its own column (resp. diagonal, antidiagonal).
So $N_{\mathrm{col}}(c)$ (resp. $N_{\mathrm{dg}}(c)$, $N_{\mathrm{ad}}(c)$) is the same for all cells $c$ of $L$.
We call $L$ \emph{odd} if this number is odd, and \emph{even} otherwise.

\medskip

\noindent\textbf{Claim A.}
\emph{Let $c$ be a cell to the right of the column $a^{\ast}$.
If $c$ lies in the row $b_{2}$ (resp. $b_{3}$, $b_{1}$), then $N_{\mathrm{col}}(c)$ (resp. $N_{\mathrm{dg}}(c)$, $N_{\mathrm{ad}}(c)$) is even.}

\noindent Let $c=(a,b_{2})$ be a cell of the row $b_{2}$ to the right of the
column $a^{\ast}$, i.e. $a>a^{\ast}$.
Let $Z$ be the closed path with vertices
\[
\begin{array}
[c]{c}
P_{1}(a,b_{2}),\quad P_{2}(a+b_{2}-b_{1},b_{1}),\quad P_{3}(a+b_{2}
+b_{3}-2b_{1},b_{3}),\\
P_{4}(a+2b_{3}-2b_{1},b_{2}),\quad P_{5}(a+2b_{3}-b_{1}-b_{2},b_{1}),\quad
P_{6}(a+b_{3}-b_{2},b_{3}),
\end{array}
\]
taken in this cyclic order, see Figure \ref{Fig03}.a).
Then the sides of $Z$ lie on diagonals and antidiagonals, and
the columns of its vertices meet the row $b_{2}$ exactly in the six cells
forming the window of $c$.

\begin{figure}[tbh]
\begin{center}
\includegraphics[scale=1]{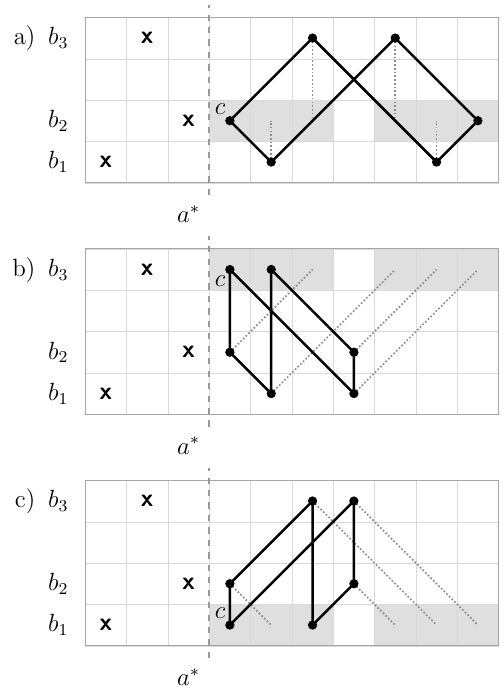}
\end{center}
\caption{The figure shows three fixed queens with rows $b_{1}<b_{2}<b_{3},$ a
chosen cell $c$ to the right of $a^{\ast},$ the window of the cell $c$
consisting of gray cells, and the zigzag line $Z$ for a) $N_{\mathrm{col}}%
(c)$, b) $N_{\mathrm{dg}}(c)$ and c) $N_{\mathrm{ad}}(c)$.}%
\label{Fig03}%
\end{figure}

Let $m$ be the number of sides of $Z$ lying on occupied lines. Every vertex of
$Z$ is an endpoint of exactly one diagonal side and of exactly one
antidiagonal side of $Z$. So, when we add up $\sigma(P_{1}),\ldots
,\sigma(P_{6})$, every side of $Z$ on an occupied line is counted exactly
twice, once at each of its endpoints, and the occupied columns through the
vertices contribute $N_{\mathrm{col}}(c)$, since the columns of the vertices
meet the row $b_{2}$ exactly in the window of $c$. Thus
\[
\sigma(P_{1})+\sigma(P_{2})+\cdots+\sigma(P_{6})=N_{\mathrm{col}}(c)+2m.
\]
All vertices of $Z$ lie in the rows $b_{1}$, $b_{2}$, $b_{3}$ to the right of
the column $a^{\ast}$, so none of them is a queen, yet each of them lies in an
occupied row. Hence every $\sigma(P_{i})$ is even, and so is
$N_{\mathrm{col}}(c)$. The same counting applied to the zigzags of Figure
\ref{Fig03}.b) and \ref{Fig03}.c) shows that $N_{\mathrm{dg}}(c)$ and
$N_{\mathrm{ad}}(c)$ are even as well. This proves Claim A.

\medskip

We remark that although the parity of $N_{\mathrm{col}}(c)$ (resp. $N_{\mathrm{dg}}(c)$, $N_{\mathrm{ad}}(c)$) was proved for row $b_2$ (resp. $b_3$, $b_1$), moving the cells on columns (resp. diagonals, antidiagonals) we can transform the cells to an arbitrary row.
Hence, by Claim~A we have that for every row $b$ occupied by a queen, there is $a_b$ such that every cell $(a,b)$ lies only on even lines whenever $a\ge a_b$.
This fact will be developed at the end of the proof.

Observe also that the set of cells of windows contains less than $6$ cells exactly when $b_3-b_2=b_2-b_1$.
In such a case windows have just 4 cells as sets, though they have 6 cells as multisets.
Nevertheless, in this case the proof of Claim~A is almost identical, just the second and third (and the fourth and fifth) cells of the window coincide, so two pairs of vertices of $Z$ lie in a common column (resp. diagonal, antidiagonal).

\medskip

\noindent\textbf{Claim B.}
\emph{Let $(a,b)\in S$ and let $c$ be a cell of the row $b$ which is either the queen $(a,b)$ or lies to the right of it.
Then $c$ lies on an odd number of odd lines if and only if $c=(a,b)$.}

\noindent Let $c_{1}=c,c_{2},\ldots,c_{6}$ be the cells of the window of $c$.
Notice that the cell $c$ lies on an odd number of odd lines if and only if
$N_{\mathrm{col}}(c)+N_{\mathrm{dg}}(c)+N_{\mathrm{ad}}(c)$ is odd. Every cell
$c_{i}$ lies in exactly one column, one diagonal and one antidiagonal, so by
the definitions of $\sigma$ and of the three numbers $N_{\mathrm{col}}$,
$N_{\mathrm{dg}}$, $N_{\mathrm{ad}}$ we have
\[
N_{\mathrm{col}}(c)+N_{\mathrm{dg}}(c)+N_{\mathrm{ad}}(c)=\sigma(c_{1}
)+\sigma(c_{2})+\cdots+\sigma(c_{6}).
\]
If $c=(a,b),$ then $\sigma(c)$ is odd while $\sigma(c_{i})$ is even for every
$i\in\{2,\ldots,6\},$ so the right-hand side is odd and $c$ lies on an odd
number of odd lines. If $c$ is not the queen, then it lies to the right of it,
so $\sigma(c_{i})$ is even for every $i\in\{1,\ldots,6\},$ the right-hand side
is even, and $c$ does not lie on an odd number of odd lines. This proves the
claim.

\medskip

\noindent\textbf{Claim C.}
\emph{Let $(a,b)\in S$ and let $c$ be a cell of the row $b$ which lies on exactly one odd line, such that no cell of the row $b$ to the right of $c$ lies on an odd line.
Then $c=(a,b)$.}

\noindent Indeed, if the queen $(a,b)$ lies to the right of $c$, then the
queen lies on no odd line, and if $c$ lies to the right of the queen, then $c$
lies on exactly one odd line. Both contradict Claim B.

\medskip

We are now ready to finish the proof. By Claim B, every queen of $S$ lies on
an odd line, so odd lines exist. Suppose first that there are odd diagonals.
By Claim A, every odd diagonal intersects the row $b_{3}$, and every odd
antidiagonal intersects the row $b_{1}$, at column $a^{\ast}$ or to the left
of it.
So there is the rightmost odd diagonal $D$, and there is a row $r>b_{1}$, such that in every row above $r$ the cell of $D$ lies to the right of the column $a^{\ast}$, and in every row above $b_{1}$ each odd antidiagonal intersects that row to the left of the column $a^{\ast}$.
Let $(a,b)\in S$ be a queen above the row $r$, and let $c$
be the cell of $D$ in the row $b$. The cell $c$ and all
cells of the row $b$ to the right of $c$ lie to the right of the column
$a^{\ast}$, hence on no odd column and on no odd antidiagonal. Also, none of
them except $c$ lies on an odd diagonal. By Claim C, we obtain $c=(a,b)$, i.e.
the queen $(a,b)$ lies on $D$. Since $S$ has queens in arbitrarily high rows,
infinitely many queens of $S$ lie on $D$, which contradicts the independence
of $S$.

Suppose next that there are no odd diagonals, but there are odd columns, and
let $C$ be the rightmost one. Every odd antidiagonal meets the row $b_{1}+g$,
for $g\geq0$, at least $g$ columns to the left of the column $a^{\ast}$. So
there is a row $r$ such that in every row above $r$ all cells lying on odd
antidiagonals are to the left of $C$. Let $(a,b)\in S$ be a queen above the
row $r$ and let $c$ be the cell of $C$ in the row $b$. Then $c$ lies on
exactly one odd line, and the cells of the row $b$ to the right of $c$ lie on
no odd line. By Claim C, the queen $(a,b)$ lies in $C$. Again, infinitely
many queens lie in the column $C$, a contradiction.

Finally, suppose that all odd lines are antidiagonals, and let $E$ be the
rightmost odd antidiagonal. Let $(a,b)\in S$. By Claim B, the queen $(a,b)$
lies on an odd antidiagonal, so $E$ meets the row $b$ in a cell $c$ which is
the queen or lies to the right of it. The cell $c$ lies on exactly one odd
line, and the cells of the row $b$ to the right of $c$ lie on no odd line. By
Claim C, the queen $(a,b)$ lies on $E$.
Thus, infinitely many queens of $S$ lie on the same antidiagonal $E$, which again contradicts the independence of $S$.
\end{proof}

Theorems~{\ref{Tm_nemaKonacnog}} and {\ref{Tm_nemaBeskonacnog}} together yield
the following corollary, which answers the question of~\cite{CPST2} for the
infinite board.

\begin{corollary}
\label{Cor_rijesenProblem}
The infinite Queen graph, on the quarter plane as well as on the whole plane,
has $\alpha_{\mathrm{od}}=1$.
\end{corollary}

{\sloppy\bigskip\noindent\textbf{Acknowledgments.}~~The authors acknowledge
the partial support by Slovak research grants VEGA 1/0011/25, VEGA 1/0069/23,
APVV-23-0076 and APVV-22-0005, by ARIS projects J1-3002 and J1-70016, program
P1-0383, bilateral Slovenian-Croatian project BI-HR/25-27-004 and the annual
work program of Rudolfovo, by Croatian Ministry of Science, Education and
Youth through the bilateral Croatian-Slovenian project 2025-26, and by the
NextGeneration EU foundation via IP-UNIST-17 (GEORAZ). }

\medskip

{\noindent\textbf{AI declaration.}~~}The results of this paper were obtained
with the assistance of Claude (Anthropic). The authors carefully verified
all arguments, worked out the presentation of the proofs, and take full responsibility
for the content of the paper.

\end{document}